\documentclass{article}

\usepackage{amsfonts}
\usepackage{amsmath}
\usepackage{diagrams}
\usepackage{fullpage}
\usepackage{color}
\usepackage{pdflscape}
\usepackage{appendix}

\newtheorem{theorem}{Theorem}

\newenvironment{proof}[1][Proof]{\noindent\textbf{#1.} }{\ \rule{0.5em}{0.5em}}

\input ArtNouvc.fd
\newcommand*\initfamily{\usefont{U}{ArtNouvc}{xl}{n}}

\begin{document}

\title{Quantizations of group actions}
\author{Hilja L. Huru, Valentin V. Lychagin}
\date{}
\maketitle

\begin{abstract}
We describe quantizations on monoidal categories of modules over finite groups. They are given by quantizers which are elements of a group algebra. Over the complex numbers we find these explicitly. For modules over $S_3$ and $A_4$ we are given explicit forms for all quantizations.
\end{abstract}

\section*{Introduction}
In the papers \cite{h2,HL} we found the quantizations of the monoidal categories of modules graded by finite abelian groups. Quantizations are natural isomorphisms of the tensor bifunctor $Q:\otimes \rightarrow \otimes$ that satisfy the coherence condition. By this condition the quantizations are 2-cocycles, and under action by isomorphisms of the identity functor they are representatives of the second cohomology of the group. 

With these explicit descriptions of quantizations we showed that classical non-commutative algebras like the quaternions and the octonions are obtained by quantizing a required number of copies of $\mathbb{R}$. Moreover we obtained new classes of non-commutative algebras. The resulting non-commutativity is governed by a braiding which is the quantization of the twist.

In this paper we will investigate the situation for quantizations of modules with action of finite groups that are not necessarily abelian. Further the definition of quantizations is widened as they may be natural transformations, not only natural isomorphisms.

Quantizations $Q$ of the monoidal category of modules over finite groups G are realized by elements $q_Q$ in the group algebra $\mathbb{C}[G\times G]$ called quantizers. These satisfy a form of the coherence condition, normalization and invariance with respect to $G$-action.

Let $\hat G$ be the dual of $G$. We use the Fourier transform to reconsider these quantizers as sequences of operators $\hat q_{\alpha,\beta}$ in $End(E_\alpha\otimes E_\beta)$ where  $E_\alpha, E_\beta$ are irreducible representations corresponding to $\alpha,\beta\in\hat G$. The coherence condition for the operators $\hat q_{\alpha,\beta}$ gives a system of quadratic matrix equations.

There is an equivalence relation on these quantizers by the action by natural isomorphisms of the identity functor. Taking the orbits of this action we arrive at the final expressions of the quantizers.

We apply the inverse of the Fourier transform to move back to the group algebra where we now have the quantizers $q_Q$ in $\mathbb{C}[G\times G]$ realizing the non-trivial quantizations in the category. 

To sum up, the procedure is as follows
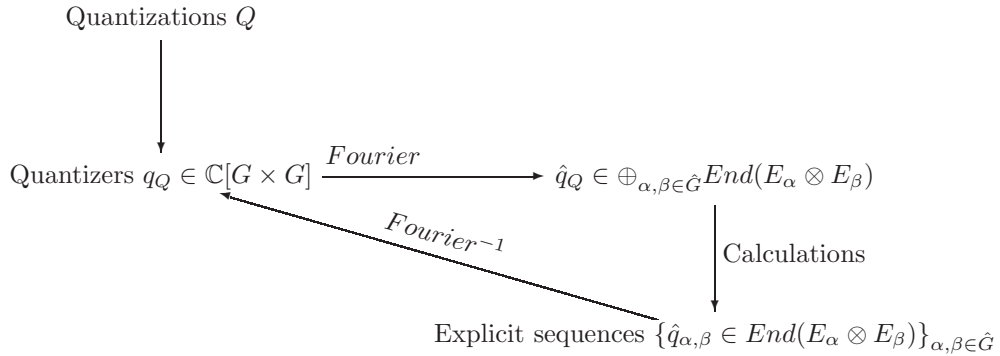
\begin{diagram}
\text{Quantizations } Q &&\\
\dTo &&\\
\text{Quantizers } q_Q\in \mathbb{C}[G\times G] & \rTo{Fourier}   & \hat q_Q\in \oplus_{\alpha,\beta\in\hat G}End(E_\alpha\otimes E_\beta)\\
 &\luTo{Fourier^{-1}} &   \dTo_{\text{Calculations}}    \\
 &&\text{Explicit sequences } \{\hat q_{\alpha,\beta}\in End(E_\alpha\otimes E_\beta)\}_{\alpha,\beta\in\hat G}
\end{diagram}

We illustrate this method for abelian groups, see e.g. \cite{h3}, and the permutation groups $S_3$ and $A_4$. 
For $S_3$ there is a 1-parameter family of quantizations. 
For $A_4$ we have a larger selection with 2 parameters producing quantizations.

\section{Quantizations in braided monoidal categories}

A braided monoidal category $C$ is a category equipped with a tensor product $\otimes$ and two natural isomorphisms: an associativity constraint $assoc_C:X\otimes(Y\otimes Z)\rightarrow (X\otimes Y)\otimes Z$ and a braiding $\sigma_C:X\otimes Y\rightarrow Y\otimes X$, for objects $X,Y,Z$ in $C$, that both satisfy MacLane coherence conditions, see \cite{maclane}.

Let $C,D$ be braided monoidal categories and $\Phi:C\rightarrow D$ be a functor.
A \textit{quantization} $Q$ of the functor $\Phi$ is a natural transformation of the tensor bifunctor
\begin{eqnarray}
Q :\otimes_{D}\circ (\Phi \times\Phi) &\rightarrow& \Phi\circ \otimes_{C} ,\\
Q =Q_{X,Y}:\Phi(X)\otimes_{D} \Phi(Y)&\rightarrow& \Phi(X\otimes_C Y),
\end{eqnarray}
that satisfies the following  coherence condition
\begin{equation}
\begin{diagram}\label{qu1}
\Phi(X)\otimes (\Phi(Y)\otimes \Phi(Z)) & \rTo{1\otimes Q}   & \Phi(X)\otimes \Phi(Y\otimes Z) &\rTo{Q}&\Phi(X\otimes (Y \otimes Z)) \\
 \dTo{assoc_{D}}  & &  &&   \dTo_{\Phi\circ assoc_C\circ\Phi^{-1}}    \\
(\Phi(X)\otimes \Phi(Y))\otimes \Phi(Z) &  \rTo{Q\otimes 1}& \Phi(X\otimes Y)\otimes \Phi(Z)   &  \rTo{Q}  & \Phi((X\otimes Y) \otimes Z)
\end{diagram}
\end{equation}
for $X,Y\in Obj(C)$ (see \cite{lq}, for details). 

Note that in this paper we do not require the quantizations to be natural isomorphisms, only natural transformations (cf. \cite{lq}). 

We denote the set of quantizations of $\Phi$ by $$\bold{q}(\Phi).$$ 

Let $\lambda:\Phi\rightarrow \Phi$ be a unit preserving natural isomorphism of the functor $\Phi$. Then we define an action of $\lambda: Q \mapsto \lambda(Q)$ on the quantizations by requiring that the diagram
\begin{equation}
\begin{diagram}\label{idfunctor}
\Phi(X)\otimes \Phi(Y)& \rTo^{Q_{X,Y}} & \Phi(X\otimes Y)\\
 \dTo^{\lambda_{X}\otimes \lambda_{Y}} && \dTo_{\lambda_{X\otimes Y}}\\
\Phi(X)\otimes \Phi(Y) & \rTo^{\lambda(Q_{X,Y})} &\Phi(X\otimes Y)
\end{diagram}
\end{equation}
commutes. 

The orbits of the above action we denote by
\begin{equation}
\text{\initfamily\selectfont Q}\normalfont(\Phi).
\end{equation}

If a quantization is a natural isomorphism we will call it a \textit{regular} quantization and use the notations ${\bold q}^\circ(\Phi)$ and $\text{\initfamily\selectfont Q}^\circ (\Phi)$.

If the functor $\Phi$ is the identity $Id_C:C\rightarrow C$ we call $Q$ a \textit{quantization of the category $C$} and use the notations ${\bold q}(C)={\bold q}(Id_C)$ and 
$\text{\initfamily\selectfont Q}\normalfont(C)=\text{\initfamily\selectfont Q}\normalfont(Id_C)$.

For the following result, see also \cite{lq}.

Let $\Phi:B\rightarrow C,\Psi:C\rightarrow D$ be functors between the monoidal categories $B,C,D$ and let $Q^\Phi, Q^\Psi$ be quantizations of $\Phi$ and $\Psi$. The following formula defines the quantization
\begin{equation}\label{composition}
Q^{\Phi\circ\Psi}=\Psi(Q_{X,Y})\circ Q^\Psi_{\Phi(X),\Phi(Y)}
\end{equation}
of the composition $\Psi\circ\Phi$. 

We call $Q^{\Phi\circ\Psi}$ composition of quantizations. Then the composition defines an associative multiplication
\begin{equation}
{\bold q}(\Phi)\times {\bold q}(\Psi)\rightarrow {\bold q}(\Psi\circ\Phi)
\end{equation}
on the sets of quantizations.

In partucular, the composition (\ref{composition}), where $\Phi=\Psi=Id_C$, defines a multiplication
\begin{equation}
{\bold q}(C)\times {\bold q}(C)\rightarrow {\bold q}(C)
\end{equation}
and gives a monoid structure on the set of quantizations.

Moreover, the regular quantizations form a group in this monoid.

The monoid ${\bold q}(C)$ acts on a variety of objects, see e.g. \cite{lq}. We list some examples here.
\begin{description}
\item[Braidings.] Let $\sigma$ be a braiding in the category $C$. If $Q\in {\bold q}^\circ(C)$ then we define an action of $Q$ on braidings by requiring that the following diagram
\begin{equation*}
\begin{diagram}
X\otimes Y & \rTo{Q_{X,Y}}& X\otimes Y\\
 \dTo{\sigma_{Q}}  & &   \dTo_{\sigma}    \\
Y \otimes X& \rTo{Q_{Y,X}} &Y\otimes X
\end{diagram}
\end{equation*}
commutes, where $$Q:\sigma\mapsto\sigma_Q=Q_{Y,X}^{-1}\circ\sigma\circ Q_{X,Y}.$$
Then $\sigma_{Q}$ is a braiding too.

\item[Algebras.] Let A be an associative algebra in the category with multiplication $\mu:A\otimes A\rightarrow A$. Let $Q\in {\bold q}(C)$. Then we define a new multiplication $\mu^Q$ on A by requiring that the following diagram\begin{equation}
\begin{diagram}
A\otimes A & \rTo{Q_{A,A}}& A\otimes A\\
 & \rdTo{\mu^Q} &   \dTo_{\mu}\\
& &A
\end{diagram}
\end{equation}
commutes.

Then $(A,\mu^Q)$ is an algebra in the category C too. We call it the \textit{quantized algebra}.

\item[Modules.] Let E be a left module over the algebra A in the category with multiplication $\nu:A\otimes E\rightarrow E$. We define a quantized multiplication $\nu^Q$ on E by requiring that the following diagram
\begin{equation}
\begin{diagram}
A\otimes E & \rTo{Q_{A,E}}& A\otimes E\\
 & \rdTo{\nu^Q} &   \dTo_{\nu}\\
& &E
\end{diagram}
\end{equation}
commutes. 

Then $(E,\nu^Q)$ is a module over the quantized algebra $(A,\mu^Q)$. We call it the \textit{quantized module}. 

Right modules are quantized in a similar way.

\item[Coalgebras.]Let $A^*$ be a coalgebra in the category with comultiplication $\mu^*:A^*\rightarrow A^*\otimes A^*$. We define a new comultiplication $(\mu^*)^Q$ on $A^*$ by requiring that the following diagram
\begin{equation}
\begin{diagram}
A^*\otimes A^* & \rTo{Q_{A^*,A^*}}& A^*\otimes A^*\\
 & \luTo{(\mu^*)^Q} &   \uTo_{\mu^*}    \\
& &A^*
\end{diagram}
\end{equation}
commutes. 

Then $(A^*,(\mu^*)^Q)$ is a coalgebra in the category C too. We call it the \textit{quantized coalgebra}.

\item[Bialgebras.] Let $B$ be a bialgebra in the category with multiplication $\mu_B:B\otimes B\rightarrow B$ and comultiplication $\mu^*_{B}:B\rightarrow B\otimes B$. The \textit{quantized bialgebra} is the same object $B_Q=B$ equipped with the quantized multiplication $\mu_B^Q$ and quantized comultiplication $(\mu^*_{B})^Q$, quantized as above. This is also a bialgebra in the category.
\end{description}

\section{Quantizations of G-modules}
Let R be a commutative ring with unit and let G be a finite group. Denote by $Mod_R(G)$ the monoidal category of finitely generated G-modules over R and let $\bold{q}(G)$ and $\text{\initfamily\selectfont Q}\normalfont(G)$ be the sets of quantizations of this category.

Let X and Y  be G-modules. Recall that the tensor product $X\otimes Y$ over R is a G-module with action
\begin{equation*}
g(x\otimes y)=g(x)\otimes g(y),
\end{equation*} 
for $g\in G, x\in X, y\in Y$. 

Let $R[G]$ be the group algebra of G over R.

Then there is an isomorphism between the categories of G-modules and $R[G]$-modules, $Mod_R(G)=Mod_R(R[G])$ (see e.g. \cite{h3}). 

Hence, $$\bold{q}(G)=\bold{q}(Mod_R(R[G])),  \text{\initfamily\selectfont Q}\normalfont(G)=\text{\initfamily\selectfont Q}\normalfont(Mod_R(R[G])).$$

\begin{theorem}\label{th-q}
Any quantization $Q\in \bold{q}(G)$ of the category of G-modules has the form
\begin{eqnarray}\label{q}
Q_{X,Y}:x\otimes y \mapsto q_Q\cdot(x\otimes y)=\sum_{g,h\in G}q_{g,h} gx\otimes hy,
\end{eqnarray}
where $$q_Q=\sum_{g,h\in G}q_{g,h} (g,h)$$
are elements of the group algebra $R[G\times G]$, for  $x\in X,y\in Y$ and $X,Y\in Obj(Mod_R(G))$.
\end{theorem}

\begin{proof}
We identify elements $x\in X$ for $X\in Obj(Mod_R(G))$ with morphisms
\begin{eqnarray*}
\phi_x:R[G]\rightarrow X,
h\mapsto h\cdot x.
\end{eqnarray*}
Then, for any elements  $x\in X,y\in Y$, $X,Y\in Obj(Mod_R(G))$ and a quantization $Q\in \bold{q}(G)$ the following diagram
\begin{equation*}
\begin{diagram}
X\otimes Y & \rTo{Q_{X,Y}}& X\otimes Y\\
 \uTo{\phi_x\otimes \phi_y}  & &   \uTo_{\phi_{x\otimes y}}    \\
R[G]\otimes R[G]& \rTo{Q_{R[G],R[G]}} &R[G]\otimes R[G]
\end{diagram}
\end{equation*}
commutes. 

Therefore,
\begin{equation*}
\begin{diagram}
x\otimes y & \rTo{Q_{X,Y}}& Q_{X,Y}(x\otimes y)=q_Q\cdot (x\otimes y)\\
 \uTo{\phi_x\otimes \phi_y}  & &   \uTo_{\phi_{x\otimes y}}    \\
1\otimes 1& \rTo{Q_{R[G],R[G]}} & q_Q
\end{diagram}
\end{equation*}
where $q_Q=Q_{R[G],R[G]}(1\otimes 1)\in R[G\times G]$, and
\begin{equation}\label{R-q}
Q_{X,Y}(x\otimes y)= q_Q\cdot (x\otimes y).
\end{equation} 
\end{proof}

We call the elements $q_Q$ \textit{quantizers} and denote by $$\bold{q}(G)=\{q_Q\in R[G\times G]\}$$ 
the set of quantizers.

\begin{theorem}
An element $q\in R[G\times G]$ defines a quantization on the category of G-modules if and only if it satisfies the following conditions:
\begin{itemize}
\item The coherence condition 
\begin{equation}\label{coh1}
(1\otimes \Delta)(q)\cdot (1\otimes q)=(\Delta\otimes 1)(q)\cdot(q \otimes 1)
\end{equation}
where $\Delta$ is the diagonal of $R[G]$. 
\item The normalization condition
\begin{equation}\label{norm}
q\cdot(x\otimes 1)=q\cdot(1\otimes x)=1.
\end{equation}
\item The naturality condition
\begin{equation}\label{Ginv}
q\cdot\Delta(g)=\Delta(g)\cdot q
\end{equation}
for $g\in G$.
\end{itemize}
\end{theorem}
\begin{proof}
The coherence condition follows from (\ref{qu1}) where $$Q_{XY,Z}:(X\otimes Y)\otimes Z \rightarrow (X\otimes Y)\otimes Z$$ 
is represented as follows
$$(x\otimes y)\otimes z \rightarrow (\Delta\otimes 1)(q)((x\otimes y)\otimes z)=\sum_{g,h\in G} q_{g,h}(gx\otimes gy)\otimes hz,$$ 
and similarly for $Q_{X,YZ}$. 

The two other conditions follow straight foreward from the normalization and naturality conditions on the quantizations.
\end{proof}

See also \cite{lq} for similar settings.

We shall from now on use the notion quantizer instead of quantization.

Remark that the conditions (\ref{coh1}), (\ref{norm}) and (\ref{Ginv}) are some kind of 2-cyclic condition on $\bold{q}(G)$ (see, for example, section \ref{abgp} on finite abelian groups).

Let $U(G)$ be the set of units of $R[G]$.
\begin{theorem}
The set of quantizers of G-modules 
$\text{\initfamily\selectfont Q}\normalfont(G)$ is the orbit space of the following $U(G)$-action on ${\bold q}(G)$:
\begin{equation}\label{ql-action}
l(q)\stackrel{\text{\tiny def}}{=}\Delta(l)\cdot q\cdot l^{-\otimes 2},
\end{equation}
where $l\in U(G)$.
\end{theorem}
\begin{proof}
Representing as above elements $x\in X$ by morphisms $\phi_x:R[G]\rightarrow X$ we get the following commutative diagram with $\lambda_X:X\rightarrow X$,
\begin{equation*}
\begin{diagram}
X & \rTo{\lambda_{X}}& X\\
\uTo{\phi_x} & & \uTo_{\phi_x}    \\
R[G]& \rTo{\lambda_{R[G]}} & R[G]
\end{diagram}
\end{equation*}
for any unit preserving natural isomorphism of the identity functor, $\lambda:Id_{Mod_R(R[G])}\rightarrow Id_{Mod_R(R[G])}$.

Therefore $\lambda$ is uniquely defined by elements $l\in \lambda_{R[G]}(1)$, and
 
\begin{equation}
\lambda_X(x)=\lambda_X(\phi_x(1))=\phi_x(\lambda_{R[G]}(1))=\phi_x(l)=l\cdot x.
\end{equation}

Let $q\in {\bold q}(G)$. Then the action (\ref{idfunctor}) gives
\begin{equation}
\lambda(q)=\Delta(l)\cdot q\cdot(l^{-1}\otimes l^{-1}).
\end{equation}
\end{proof}

We say that two quantizers $p,q\in{\bold q}(G)$ are equivalent if $p=l(q)$ for some $l\in U(G)$.

\section{The Fourier transform}

In this section we'll use the Fourier transform to find the quantizers, under the assumption that $R=\mathbb{C}$.
 
Below we list necessary formulae from representation theory of groups (see for example \cite{simon}).

Denote by $\hat G$ the dual of $G$ which is the set of equivalence classes of the irreducible representations of $G$. 
For each $\alpha\in\hat G$ we pick the corresponding irreducible representation on $E_\alpha$, $dim E_\alpha=d_\alpha$, and an explicit realisation of this representation by a $d_\alpha\times d_\alpha$-matrix $D^\alpha(g)=|D_{ij}^\alpha(g)|:E_\alpha\rightarrow E_\alpha$. 

The elements $D_{ij}^\alpha(g)\cdot g\in \mathbb{C}[G]$ span the group algebra $\mathbb{C}[G]$ and $\mathbb{C}[G]$ is isomorphic as an algebra to a direct sum of matrix algebras by the Fourier transform 
\begin{eqnarray*}
F:\mathbb{C}[G]&\rightarrow& \oplus_{\alpha\in\hat G}End(E_\alpha)\\
\tilde{ f}=\sum_{g\in G}f(g)\cdot g &\mapsto& \hat f = \{\hat f_\alpha\}_{\alpha\in\hat G}.
\end{eqnarray*}
We'll consider $\hat f$ as a ``function'' on the dual group which at each point $\alpha\in\hat G$ takes values in $End(E_\alpha)$: 
$$\hat f_\alpha =\hat f(\alpha)\in End(E_\alpha),$$ 
and
\begin{equation}
F(\tilde f)(\alpha)=\hat f_\alpha=\sum_{g\in G}D^\alpha(g)f(g).
\end{equation}
The inverse of the Fourier transform has the following form
\begin{eqnarray}
F^{-1}(\hat f)&=&\frac{1}{|G|}\sum_{g\in G}\sum_{\alpha\in \hat G} d_{\alpha} Tr(D^{\alpha}(g)^*\hat f_{\alpha})\cdot g.
\end{eqnarray}

As we have seen the quantizers are elements of $\mathbb{C}[G\times G]$. 

The dual of $G\times G$ is $\widehat{G\times G}=\hat G\times\hat G$, and
$$E_{\alpha\beta}=E_\alpha\otimes E_\beta$$
for $(\alpha,\beta)\in\hat G\times \hat G$ with the action 
$$D^{\alpha,\beta}(g,h)=D^{\alpha}(g)\otimes D^{\beta}(h).$$
In this case the Fourier transform 
$$F:\mathbb{C}[G]\otimes \mathbb{C}[G]=\mathbb{C}[G\times G]\rightarrow \oplus_{\alpha,\beta\in\hat G}End(E_\alpha\otimes E_\beta)$$ 
and its inverse have the following forms
\begin{eqnarray}
F(\tilde f){(\alpha,\beta)}=\hat f_{\alpha,\beta}&=&\sum_{g,h\in G}f(g,h){D^{\alpha,\beta}(g,h)}\\
\label{invF}
F^{-1}(\hat f)&=&\frac{1}{|G|^2}\sum_{g,h\in G}\sum_{\alpha,\beta\in \hat G} d_{\alpha,\beta} Tr(D^{\alpha,\beta}(g,h)^*\hat f_{\alpha,\beta})\cdot(g,h)
\end{eqnarray}
where 
$$\tilde f=\sum_{g,h\in G}f(g,h)(g,h)$$ 
and 
$$\{\hat f_{\alpha,\beta}\in End(E_{\alpha}\otimes E_\beta)\}_{\alpha,\beta\in\hat G}.$$

Let $\chi_\alpha(g)=Tr(D^\alpha(g))$ be the character of the irreducible representation $E_\alpha,\alpha\in\hat G$.

Splitting of the tensor product of $E_\alpha\otimes E_\beta$ into a sum of irreducible representations we get isomorphisms  
\begin{equation}\label{tensor}
\nu_{\alpha,\beta}:E_\alpha\otimes E_\beta\rightarrow\oplus_{\gamma\in \hat G}c_{\alpha\beta}^\gamma E_\gamma,
\end{equation} 
where $c_{\alpha\beta}^\gamma\in\mathbb{N}$ are called the \textit{Clebsch-Gordan integers}. 

These integers can be computed as follows
$$\chi_\alpha\cdot\chi_\beta=\sum_\gamma c_{\alpha\beta}^\gamma\chi_\gamma.$$ 

Projections $p_\alpha$ of G-modules $E=\sum_{\alpha\in\hat G}c^{\alpha} E_\alpha$ onto its irreducible components $c_\alpha E_\alpha=E_{(\alpha)}$ are the following
\begin{eqnarray}
p_\alpha=\frac{d_\alpha}{|G|}\sum_{g\in G}\chi_\alpha(g) D^\alpha(g^{-1}).
\end{eqnarray}
where $E_{(\alpha)}\simeq\mathbb{C}^{c^\alpha}\otimes E_{\alpha}$.
They satisfy orthogonality conditions
\begin{eqnarray}
\sum_{\alpha\in\hat G}p_\alpha =1,p_\alpha^2=p_\alpha,p_\alpha p_\beta=0.
\end{eqnarray}

For the tensor products (\ref{tensor}) the projectors take the form
\begin{equation}
p_\gamma=\sum_{g\in G}\chi_\gamma(g)D^\alpha(g^{-1})\otimes D^\beta(g^{-1}).
\end{equation}
The matrix representation is
$D^\alpha(g)\otimes D^\beta(h)=|D^{\alpha\beta}_{\gamma}(g,h)|$, $g,h\in G$ where
\begin{equation}
\mathbb{C}^{c_{\alpha\beta}^{\gamma}}\otimes E_{\gamma} \rTo{D^{\alpha\beta}_{\gamma}(g,h)} \mathbb{C}^{c_{\alpha\beta}^{\gamma}}\otimes E_{\gamma}.
\end{equation}

\subsection{The Fourier transform on quantizers}

We now rewrite the coherence condition (\ref{qu1}) for quantizers in terms of their Fourier transforms.

Let $q=\sum_{g,h\in G}q_{g,h}(g,h)$ be a quantizer and 
$$F(q)=\hat q=\sum_{\alpha,\beta\in\hat G}\hat q_{\alpha,\beta}$$ 
where $q_{\alpha,\beta}\in End(E_{\alpha}\otimes E_\beta )$. 
Then 
\begin{eqnarray}
\hat q(\alpha,\beta)=\hat q_{\alpha,\beta} = \sum_{g,h\in G}q(g,h){D^{\alpha,\beta}(g,h)}.
\end{eqnarray}
The operators $\hat q_{\alpha,\beta}:E_\alpha\otimes E_\beta\rightarrow E_\alpha\otimes E_\beta$ are G-morphisms. 

Therefore, due to isomorphisms $\nu_{\alpha,\beta}$, each $\hat q_{\alpha,\beta}$ is a direct sum 
\begin{equation}\label{q-decomp}
\hat q_{\alpha,\beta}=\oplus_{\gamma\in \hat G}\hat q^\gamma_{\alpha,\beta}
\end{equation}
of operators  
$$\hat q_{\alpha,\beta}^\gamma:c_{\alpha\beta}^\gamma E_\gamma\rightarrow c_{\alpha\beta}^\gamma E_\gamma.$$

Note that the operators $\hat q_{\alpha,\beta}^\gamma$ are given by $c_{\alpha\beta}^\gamma\times c_{\alpha\beta}^\gamma$-matrices. 

Rewriting the coherence condition in terms of these operators we get the following result.

\begin{theorem}\label{expandQ}
Let $q$ be a quantizer on the monoidal category $Mod_\mathbb{C}(G)$. Then the coherence condition diagram (\ref{qu1}) under the Fourier transform take the following form
\begin{equation}
\begin{diagram}
\sum_{\eta,\zeta}c^\zeta_{\alpha\eta} c^\eta_{\beta\gamma} E_\zeta& \rTo{} & \sum_{\eta,\zeta}c^\zeta_{\alpha\eta} c^\eta_{\beta\gamma} E_\zeta&  \rTo{\sum_{\zeta,\eta} \hat q^\zeta_{\alpha,\eta}} &\sum_{\eta,\zeta}c^\zeta_{\alpha\eta} c^\eta_{\beta\gamma} E_\zeta\\
\uTo{\nu_{\alpha\eta}}  & & \uTo{\nu_{\alpha\eta}}& &      \uTo{\nu_{\alpha\eta}}    \\
E_\alpha\otimes (\sum_\eta c^\eta_{\beta\gamma} E_\eta)   &  \rTo{1\otimes\sum_\eta \hat q^\eta_{\beta,\gamma}}  &E_\alpha\otimes (\sum_\eta c^\eta_{\beta\gamma} E_\eta) &  \rTo{\sum_\eta \hat q_{\alpha,\eta}} & E_\alpha\otimes (\sum_\eta c^\eta_{\beta\gamma} E_\eta)\\
 \uTo{\nu_{\beta\gamma}}  & & \uTo{\nu_{\beta\gamma}} & &   \uTo{\nu_{\beta\gamma}}   \\ 
 E_\alpha\otimes (E_\beta\otimes E_\gamma) & \rTo{1\otimes \hat q_{\beta,\gamma}}   & E_\alpha\otimes (E_\beta\otimes E_\gamma) & \rTo{\hat q_{\alpha,\beta \gamma}}  &  E_\alpha\otimes (E_\beta\otimes E_\gamma)\\
 \dTo{assoc}  & & & &     \dTo{assoc}    \\
(E_\alpha\otimes E_\beta)\otimes E_\gamma   &  \rTo{\hat q_{\alpha,\beta}\otimes 1}  & (E_\alpha\otimes E_\beta)\otimes E_\gamma&  \rTo{\hat q_{\alpha \beta,\gamma}} & (E_\alpha\otimes E_\beta)\otimes E_\gamma\\
 \dTo{\nu_{\alpha\beta}}  & & \dTo{\nu_{\alpha\beta}}& &      \dTo{\nu_{\alpha\beta}}    \\
(\sum_\xi c^\xi_{\alpha\beta} E_\xi)\otimes E_\gamma   &  \rTo{\sum_\xi \hat q^\xi_{\alpha,\beta}\otimes 1}  & (\sum_\xi c^\xi_{\alpha\beta} E_\xi)\otimes E_\gamma&  \rTo{\sum_\xi \hat q_{\xi,\gamma}} & (\sum_\xi c^\xi_{\alpha\beta} E_\xi)\otimes E_\gamma\\
\dTo{\nu_{\xi\gamma}}  & & \dTo{\nu_{\xi\gamma}} & &   \dTo{\nu_{\xi\gamma}}   \\ 
\sum_{\xi,\zeta}c^\zeta_{\xi\gamma} c^\xi_{\alpha\beta} E_\zeta&  \rTo & \sum_{\xi,\zeta}c^\zeta_{\xi\gamma} c^\xi_{\alpha\beta} E_\zeta&  \rTo{\sum_{\zeta,\xi} \hat q^\zeta_{\xi,\gamma}} & \sum_{\xi,\zeta}c^\zeta_{\xi\gamma} c^\xi_{\alpha\beta} E_\zeta
\end{diagram}
\end{equation}
where $\hat q_{\alpha,\beta\gamma}\in End(E_\alpha \otimes(E_\beta\otimes E_\gamma))$ and $\hat q_{\alpha\beta,\gamma}\in End((E_\alpha\otimes E_\beta)\otimes E_\gamma)$.
\end{theorem}
Assuming that our category is strict we get the following conditions for the quantizers.
\begin{theorem}\label{Qfinal}
 The set of operators $\hat q_{\alpha,\beta}^\gamma\in End(c_{\alpha\beta}^\gamma E_\gamma,c_{\alpha\beta}^\gamma E_\gamma)$ defines a quantizer 
\begin{equation}
q=\frac{1}{|G|^2}\sum_{g,h\in G} \sum_{\alpha,\beta\in \hat G} d_{\alpha,\beta}\oplus_{\gamma\in\hat G} Tr(D^{\alpha,\beta}_{\gamma}(g,h)^*\hat q_{\alpha,\beta}^\gamma)\cdot(g,h)\in\mathbb{C}[G\times G]
\end{equation} 
if and only if these operators are solutions of the following system of quadratic equations:
\begin{eqnarray}\label{coh}
&\sum_{\eta,\zeta\in \hat G} \hat q^\zeta_{\alpha,\eta}\hat q^\eta_{\beta,\gamma} =
\sum_{\zeta,\xi\in \hat G} \hat q^\zeta_{\xi,\gamma}\hat q^\xi_{\alpha,\beta}\\
&\hat q_{\alpha,0}^\alpha=\hat q_{0,\alpha}^\alpha=1
\end{eqnarray}
for all $\alpha,\beta, \gamma \in \hat G$.
\end{theorem}
\begin{proof}
The first condition follows from theorem \ref{expandQ}. 

The second condition is the normalization condition where $\hat q_{0,\alpha}=\hat q_{0,\alpha}^\alpha=1:1\otimes E_\alpha=E_\alpha\rightarrow 1\otimes E_\alpha=E_\alpha$ and $\hat q_{\alpha,0}=\hat q_{\alpha,0}^\alpha=1:E_\alpha\otimes 1=E_\alpha\rightarrow E_\alpha\otimes 1=E_\alpha$
\end{proof}

Let $F(l)=\sum_{\alpha\in\hat G}\hat l_\alpha$ be the Fourier transform of $l\in U(G)$ where $\hat l_\alpha\in\mathbb{C}^*$ due to Shur lemma, and $l_0=1$. Then action (\ref{ql-action}) can be rewritten as follows:
\begin{equation}\label{l-action}
l(\hat q_{\alpha,\beta})=\oplus_{\gamma\in \hat G} \hat l_\gamma \hat l_\alpha^{-1} \hat l_\beta^{-1}\hat  q_{\alpha,\beta}
\end{equation}
where $\hat l_\alpha,\hat l_\beta,\hat l_\gamma\in\mathbb{C}^*$.

\section{Finite abelian groups}\label{abgp}

Let $G$ be a finite abelian group and $R=\mathbb{C}$. 

In \cite{h2,h3} we investigated regular quantizations of modules
with action and coaction by finite abelian groups. In this section we shall revisit this case by using the Fourier transform.

By theorem \ref{th-q} the quantizations of G-modules have the form $x\otimes y\mapsto q\cdot(x\otimes y)$ for elements $x\in X,y\in Y$ in G-modules $X$ and $Y$ where $q=\sum_{g,h\in G}q_{g,h}(g,h)\in\mathbb{C}[G\times G]$.

Let $\hat G$ be the dual of $G$. All irreducible representations of G are 1-dimensional and identified with characters $\alpha\in Hom(G,\mathbb{C}^*)=\hat G$.

The Fourier transform has the form
\begin{equation*}
F(f)(\alpha)=\hat f_\alpha =\sum_{g\in G}f(g)\alpha(g^{-1})
\end{equation*}
for $f=\sum_{g\in G}f(g)g \in\mathbb{C}[G]$. The inverse of this Fourier transform is
\begin{equation*}
F^{-1}(\hat f_\alpha)=\frac{1}{|G|}\sum_{g\in G}\hat f_\alpha \alpha(g)g.
\end{equation*}
Then 
\begin{eqnarray}
F(q)(\alpha,\beta)=\hat q_{\alpha,\beta} =\sum_{g,h\in G}q_{g,h}\alpha(g^{-1})\beta(h^{-1})
\end{eqnarray}
is an operator 
$$\hat q_{\alpha,\beta}:E_\alpha\otimes E_\beta\rightarrow E_\alpha\otimes E_\beta$$ where $\alpha,\beta\in\hat G$.

Clearly $\hat q_{\alpha,\beta}\in\mathbb{C}^*$.

Corresponding to theorem \ref{Qfinal} we thus have the following conditions on $\hat q_{\alpha,\beta}$ 
\begin{eqnarray}
&\hat q_{\alpha\cdot\beta,\gamma}\hat q_{\alpha,\beta}=\hat q_{\alpha,\beta\cdot\gamma}\hat q_{\beta,\gamma}\\
&\hat q_{0,\alpha}=\hat q_{\alpha,0}=1
\end{eqnarray}
for all $\alpha,\beta,\gamma\in\hat G$ where the first condition is given by the coherence condition and the second is the normalization condition. 

Denote by $\bold{q}(\hat G)$ the group of all functions satisfying these conditions. 
We see that they are 2-cocycles. 

Hence  $\bold{q}(\hat G)$ is represented by the multiplicative 2-cocycles $\hat q$ on $\hat G$ with coefficients in $\mathbb{C}^*$, where 
$$\hat q(\alpha,\beta)=\hat q_{\alpha,\beta}.$$

The $\mathbb{C}^*$-action (\ref{l-action}) on the operators $\hat q_{\alpha,\beta}$ has the form 
\begin{equation}
\hat l_\alpha^{-1}\hat l_\beta^{-1} \hat q_{\alpha,\beta}\hat l_{\alpha\cdot\beta}
\end{equation}
where $\hat l_\alpha, \hat l_\beta,\hat l_{\alpha\cdot\beta}\in \mathbb{C}^*$.

Summing up, we get the following result.

\begin{theorem}
Let G be a finite abelian group. Then the group of regular quantizations $\text{\initfamily\selectfont Q}^\circ\normalfont(\hat G)$ is isomorphic to the 2nd multiplicative cohomology group $H^2(\hat G,\mathbb{C}^*)$.

Moreover, any 2-cocycle $z\in Z^2(\hat G,\mathbb{C}^*)$ defines an quantizer $q_z$ in the following way
\begin{equation}
q_z=\frac{1}{|G|^2}\sum_{g,h\in G}\sum_{\alpha,\beta\in\hat G}c(\alpha,\beta)\alpha(g)\beta(g)\cdot(g,h)\in\mathbb{C}[G\times G].
\end{equation}
\end{theorem}

\section{Quantizations of $S_3$-modules}
We consider the symmetric group $G=S_3$. Let the representatives of the orbits of the adjoint action be $(),(1,2)$ and $(1,2,3)$ and let $\chi_0$, $\chi_1$ and $\chi_2$ be the characters of the irreducible representations corresponding to these orbits. These irreducible representations are the trivial, sign and standard representations on modules $E_0,E_1$ and $E_2$ with matrix realisations $D^0, D^1$ and $D^2$ respectively.
\begin{theorem}
For $S_3$-modules over $\mathbb{C}$ the set of quantizers $\text{\initfamily\selectfont Q}\normalfont(S_3)$ consists of 
\begin{itemize}
\item the trivial quantizer $q=1$:
\item the 1-parameter family of quantizers
\begin{eqnarray}
q_a&=&1 +p\sum_{g,h\in S_3} Tr(D^{2,2}_{0}(g,h)^*)+ Tr(D^{2,2}_{1}(g,h)^*)\cdot (g,h)
\end{eqnarray}
where $p\in\mathbb{C}$, and:
\item in addition the discrete set of discrete quantizers
\begin{eqnarray}
q_b&=&1 +\sum_{g,h\in S_3} Tr(D^{2,2}_{1}(g,h)^*)+ Tr(D^{2,2}_{2}(g,h)^*)\cdot (g,h)\\
q_c&=&1 +\sum_{g,h\in S_3} Tr(D^{2,2}_{2}(g,h)^*)\cdot (g,h)\\
q_d&=&1 +\sum_{g,h\in S_3} Tr(D^{2,2}_{1}(g,h)^*)\cdot (g,h)\\
q_e&=&1 +\sum_{g,h\in S_3} Tr(D^{2,2}_{0}(g,h)^*)+ Tr(D^{2,2}_{1}(g,h)^*)\cdot (g,h).
\end{eqnarray}
\end{itemize}
The operators $D_{i}^{2,2}:E_i\rightarrow E_i,i=0,1,2$ are the components of $D^{2,2}$ corresponding to the decomposition of the tensor product $E_2\otimes E_2=E_0\oplus E_1\oplus E_2$. 
\end{theorem}
\begin{proof}
The multiplication table for the characters of $S_3$ is
\begin{eqnarray}
\begin{array}{c|c|c|c|}
\cdot &\chi_0 &\chi_1&\chi_2\\\hline
\chi_0&\chi_0&\chi_1&\chi_2\\\hline
\chi_1&\chi_1&\chi_0&\chi_2\\\hline
\chi_2&\chi_2&\chi_2&\chi_0+\chi_1+\chi_2\\\hline
\end{array}
\end{eqnarray}
and by (\ref{tensor}) we get the multiplication table for irreducible representations
\begin{eqnarray}
\begin{array}{c|c|c|c|}
\otimes &E_0 &E_1&E_2\\\hline
E_0&E_0&E_1&E_2\\\hline
E_1&E_1&E_0&E_2\\\hline
E_2&E_2&E_2&E_0\oplus E_1\oplus E_2\\\hline
\end{array}
\end{eqnarray}
where the irreducible representations $E_0,E_1,E_2$ are $1,1$ and $2$ dimensional respectively.

By (\ref{q-decomp}) the quantizers $\hat q_{ij}$ in $End(E_i\otimes E_j)$ are decomposed as follows
\begin{eqnarray*}
\hat q_{11}&=&\hat q_{11}^0,\\
\hat q_{12}&=&\hat q_{12}^2,\\
\hat q_{21}&=&\hat q_{21}^2,\\
\hat q_{22}&=&\hat q_{22}^0\oplus  \hat q_{22}^1\oplus  \hat q_{22}^2.
\end{eqnarray*}

By normalization condition
\begin{eqnarray*}
\hat q_{00}=\hat q_{01}=\hat q_{10}=\hat q_{02}=\hat q_{20}=1.
\end{eqnarray*}

Theorem \ref{Qfinal} for triple tensor products of all combinations of $E_0,E_1,E_2$ gives the following relations (see the appendix for the details of the calculations)
\begin{eqnarray}
\hat q_{12} &=& \hat q_{21}\label{eq3},\\
\hat q_{11} &=& (\hat q_{12})^2,\label{eq1}\\ 
\hat q^0_{22} &=& \hat q_{12}\hat q^1_{22}.\label{eq2}
\end{eqnarray}

The action of the group $U(S_3)$ have the following form:
\begin{eqnarray*}
\hat q_{11}&\rightarrow&\frac{\hat l_0}{(\hat l_1)^2}\cdot{\hat q_{11}}=\frac{1}{(\hat l_1)^2}\cdot{\hat q_{11}}\label{1}\\
\hat q_{12}=\hat q_{21}&\rightarrow&\frac{\hat l_2}{\hat l_1\hat l_2}\cdot{\hat q_{12}}=\frac{1}{\hat l_1}\cdot{\hat q_{12}}\\
\hat q_{22}^0&\rightarrow&\frac{\hat l_0}{(\hat l_2)^2}\cdot{\hat q_{22}^0}=\frac{1}{(\hat l_2)^2}\cdot{\hat q_{22}^0}\\
\hat q_{22}^1&\rightarrow&\frac{\hat l_1}{(\hat l_2)^2}\cdot{\hat q_{22}^1}\\
\hat q_{22}^2&\rightarrow&\frac{\hat l_2}{(\hat l_2)^2}\cdot{\hat q_{22}^2}=\frac{1}{\hat l_2}\cdot{\hat q_{22}^2}
\end{eqnarray*}
where $\hat l_0=1$ and $\hat l_1,\hat l_2\in \mathbb{C}^*$. 

If the quantizers all are nonzero we may chose $\hat l_1,\hat l_2$ such that $\hat q_{12},\hat q^2_{22}\rightarrow 1$, by (\ref{eq3}-\ref{eq1}) then also $\hat q_{11},\hat q_{21}\rightarrow 1$ and by (\ref{eq2}) $\hat q^0_{22}=\hat q^1_{22}$. We then have the following sequence of quantizers depending on one parameter $\lambda \in\mathbb{C}$
\begin{eqnarray}
\begin{array}{c|c|c|c|c|c|c|}
&\hat q_{11}&\hat q_{12}&\hat q_{21} &\hat q_{22}^0&\hat q_{22}^1&\hat q_{22}^2\\\hline
a)&1&1&1&\lambda&\lambda&1\\\hline
\end{array}
\end{eqnarray}
Equivalently, the representatives can be chosen as follows:
\begin{eqnarray}\label{eqform}
\begin{array}{c|c|c|c|c|c|c|}
&\hat q_{11}&\hat q_{12}&\hat q_{21} &\hat q_{22}^0&\hat q_{22}^1&\hat q_{22}^2\\\hline
a')&\lambda^2&\lambda&\lambda&\lambda&1&1\\\hline
a'')&\lambda^{-2}&\lambda^{-1}&\lambda^{-1}&1&\lambda&1\\\hline
a''')&1&1&1&1&1&\lambda\\\hline
\end{array}
\end{eqnarray}

If one or both of the quantizers $\hat q_{12},\hat q^2_{22}$ are equal zero then rest will either be equal to 0 or map to 1 by choosing $l_1,l_2$ appropriately. 

By the conditions (\ref{eq3}-\ref{eq2}) they vary as follows
\begin{eqnarray}
\begin{array}{c|c|c|c|c|c|c|}
&\hat q_{11}&\hat q_{12}&\hat q_{21} &\hat q_{22}^0&\hat q_{22}^1&\hat q_{22}^2\\\hline
b)&0&0&0&0&1&1\\\hline
c)&0&0&0&0&0&1\\\hline
d)&0&0&0&0&1&0\\\hline
e)&1&1&1&1&1&0\\\hline
f)&0&0&0&0&0&0\\\hline
g)&1&1&1&0&0&0\\\hline
\end{array}
\end{eqnarray}

We now apply the inverse Fourier transform (\ref{invF}) to the quantizers a)-g) and get the corresponding element $q$ in the group algebra,
\begin{eqnarray*}
q&=&\frac{1}{|S_3|^2}\sum_{g,h\in S_3}F^{-1}(\hat q)(g,h)
=\frac{1}{|S_3|^2}\sum_{g,h\in S_3} \sum_{\chi_i,\chi_j\in \hat S_3} d_{i,j} Tr(D^{i,j}(g,h)^*\hat q_{ij})\cdot(g,h)\\
&=&\frac{1}{|S_3|^2}\sum_{g,h\in S_3} (1 + sign(h)+ sign(g)+ \hat q_{11}sign(g)sign(h))\cdot(g,h)\\
&+&  \frac{2}{|S_3|^2}\sum_{g,h\in S_3} (Tr(D^{2}(h)^*)+ Tr(D^{2}(g)^*)+ \hat q_{12}sign(g)Tr(D^{2}(h)^*)+ \hat q_{21}Tr(D^{2}(g)^*)sign(h))\cdot(g,h)\\
&+&  \frac{4}{|S_3|^2}\sum_{g,h\in S_3} Tr(D^{2,2}(g,h)^*\hat q_{22})\cdot(g,h)\\
&=&1+\frac{4}{|S_3|^2}\sum_{g,h\in S_3} Tr(D_{0}^{2,2}(g,h)^*\hat q^0_{22})+ Tr(D_{1}^{2,2}(g,h)^*\hat q^1_{22})+ Tr(D_{2}^{2,2}(g,h)^*\hat q^2_{22})\cdot(g,h)
\end{eqnarray*}
where $Tr(D^{i,j}(g,h)^*\hat q_{ij})=\hat q_{ij}Tr(D^{i}(g)^*\otimes D^{j}(h)^*)=\hat q_{ij}Tr(D^{i}(g)^*)Tr( D^{j}(h)^*)$. 

Since $Tr(trivial(g)^*)$ is 1 for all $g\in S_3$; $Tr(sign(g)^*)$ is 1 for (), (1,2,3) and (1,3,2) and -1 for (1,2), (1,3) and (2,3); and $Tr(D^2(g)^*)$ is 2 for (), -1 for (1,2,3) and (1,3,2) and 0 for (1,2), (1,3) and (2,3), most of the sum cancels out.

Then for option a) we let $\lambda=1+\frac{|S_3|^2}{4}p$ for $p\in\mathbb{C}$ and get the following class of quantizers
\begin{eqnarray}
q_a&=&1 +p\sum_{g,h\in S_3} Tr(D^{2,2}_{0}(g,h)^*)+ Tr(D^{2,2}_{1}(g,h)^*)\cdot(g,h).
\end{eqnarray}

In addition to this case the combinations b)-e) give the following quantizers
\begin{eqnarray}
q_b&=&1 +\sum_{g,h\in S_3} Tr(D^{2,2}_{1}(g,h)^*)+ Tr(D^{2,2}_{2}(g,h)^*)\cdot(g,h)\\
q_c&=&1 +\sum_{g,h\in S_3} Tr(D^{2,2}_{2}(g,h)^*)\cdot(g,h)\\
q_d&=&1 +\sum_{g,h\in S_3} Tr(D^{2,2}_{1}(g,h)^*)\cdot(g,h)\\
q_e&=&1 +\sum_{g,h\in S_3} Tr(D^{2,2}_{0}(g,h)^*)+ Tr(D^{2,2}_{1}(g,h)^*)\cdot(g,h)
\end{eqnarray}
The combinations f) and g) give the trivial quantizer.
\end{proof}

Remark that from (\ref{eqform}) the quantizer $q_a$ has the following equivalent forms
\begin{eqnarray}
q_{a'}&=&1 +p\sum_{g,h\in S_3} Tr(D^{2,2}_{0}(g,h)^*)\cdot(g,h)\\
q_{a''}&=&1 +p\sum_{g,h\in S_3} Tr(D^{2,2}_{1}(g,h)^*)\cdot(g,h)\\
q_{a'''}&=& 1+p\sum_{g,h\in S_3} Tr(D^{2,2}_{2}(g,h)^*)\cdot(g,h).
\end{eqnarray}

\section{Quantizations of $A_4$}
Let $G=A_4$ be the alternating group. Elements $(),(12)(34), (123), (132)$ represent the orbits of the adjoint $G$-action and let $\chi_0$, $\chi_1$, $\chi_2$ and $\chi_3$ be the characters of the irreducible representations corresponding to these orbits. 

These irreducible representations are the trivial representation, the first and second nontrivial one-dimensional representations and the three-dimensional irreducible representation on modules $E_0,E_1,E_2$ and $E_3$ with matrix realisations $D^0, D^1,D^2$ and $D^3$ respectively.
\begin{theorem}
For $A_4$-modules the set of quantizers $\text{\initfamily\selectfont Q}\normalfont(A_4)$  consists of 

\begin{itemize}
\item the trivial quantizer $q=1,$
\item  $\,\,q_a=1+\sum_{g,h\in A_4}Tr(D_{3}^{3,3}(g,h)^*M)\cdot(g,h)$

where $M=
\left[\begin{array}{cc}
\lambda & 0\\ 
0 &\kappa
\end{array}\right]
,
\left[\begin{array}{cc}
\lambda & 1\\
0 &\lambda
\end{array}\right]$, $\lambda, \kappa\in\mathbb{C}$,
\item $\,\,
q_b=1+\sum_{g,h\in A_4}Tr(D_{3}^{3,3}(g,h)^*P)\cdot(g,h)$
\item $\,\,
q_c=1+\sum_{g,h\in A_4}(Tr(D_{1}^{3,3}(g,h)^*)+Tr(D_{3}^{3,3}(g,h)^*P))\cdot(g,h)$
\item $\,\,
q_d=1+\sum_{g,h\in A_4}(Tr(D_{2}^{3,3}(g,h)^*)+Tr(D_{3}^{3,3}(g,h)^*P))\cdot(g,h)$
\item $\,\,
q_e=1+\sum_{g,h\in A_4}(Tr(D_{1}^{3,3}(g,h)^*)+Tr(D_{2}^{3,3}(g,h)^*)+Tr(D_{3}^{3,3}(g,h)^*P))\cdot(g,h)$

where $P$ are $2\times 2$-matrices of the form
$
\left[\begin{array}{cc}
0 & 0\\
0 &0
\end{array}\right]
,
\left[\begin{array}{cc}
0 & 1\\
0 & 0
\end{array}\right]
,
\left[\begin{array}{cc}
1 & 1\\
0 & 1
\end{array}\right]
,
\left[\begin{array}{cc}
1 & 0\\ 
0 &\lambda
\end{array}\right]
,\lambda\in \mathbb{C}$.
\end{itemize}
The operators $D_{i}^{3,3},i=0,1,2,3$ are components of $D^{3,3}$ corresponding to the decomposition $E_3\otimes E_3=E_0\oplus E_1\oplus E_2\oplus 2E_3$. 
\end{theorem}
\begin{proof}
The multiplication table for the characters of $A_4$ has the form
\begin{eqnarray}
\begin{array}{c|c|c|c|c|}
\cdot &\chi_0 &\chi_1&\chi_2&\chi_3\\\hline
\chi_0&\chi_0&\chi_1&\chi_2&\chi_3\\\hline
\chi_1&\chi_1&\chi_2&\chi_0&\chi_3\\\hline
\chi_2&\chi_2&\chi_0&\chi_1&\chi_3\\\hline
\chi_3&\chi_3&\chi_3&\chi_3&\chi_0+\chi_1+\chi_2+2\chi_3\\\hline
\end{array}
\end{eqnarray}
and by (\ref{tensor}) we get the multiplication table for irreducible representations
\begin{eqnarray}
\begin{array}{c|c|c|c|c|}
\otimes &E_0 &E_1&E_2&E_3\\\hline
E_0&E_0&E_1&E_2&E_3\\\hline
E_1&E_1&E_2&E_0&E_3\\\hline
E_2&E_2&E_0&E_1&E_3\\\hline
E_3&E_3&E_3&E_3&E_0\oplus E_1\oplus E_2\oplus 2E_3\\\hline
\end{array}
\end{eqnarray}
Remind that the irreducible representations $E_0,E_1,E_2,E_3$ are 1,1,1 and 3 dimensional respectively.

By (\ref{q-decomp}) the quantizers $\hat q_{ij}$ in $End(E_i\otimes E_j)$ split as follows
\begin{eqnarray*}
\hat q_{11}&=&\hat q^2_{11}\\
\hat q_{12}&=&\hat q^0_{12}\\
\hat q_{21}&=&\hat q^0_{21}\\
\hat q_{13}&=&\hat q^3_{13}\\
\hat q_{31}&=&\hat q^3_{31}\\
\hat q_{22}&=&\hat q^1_{22}\\
\hat q_{23}&=&\hat q^3_{23}\\
\hat q_{32}&=&\hat q^3_{23}\\
\hat q_{33}&=&\hat q^0_{33}\oplus \hat q^1_{33}\oplus \hat q^2_{33}\oplus [\hat q^3_{33}],
\end{eqnarray*}
where we use the notation $[\hat q^3_{33}]$ to keep in mind that this is a $2\times2$-matrix acting on $2E_3$. 

Further, by normalization condition
\begin{eqnarray*}
\hat q_{00}=\hat q_{01}=\hat q_{10}=\hat q_{02}=\hat q_{20}=\hat q_{03}=\hat q_{30}=1.
\end{eqnarray*}
Theorem \ref{Qfinal} for triple tensor products of all combinations of  $E_0,E_1,E_2,E_3$ give the following relations (see the appendix for details of the calculations)
\begin{eqnarray}
\hat q_{12}&=&\hat q_{21}=\hat q_{11}\hat q_{22}\label{cond1}\\
\hat q_{13}&=&\hat q_{31}\\
\hat q_{23}&=&\hat q_{32}\\
(\hat q_{13})^2&=&\hat q_{11}\hat q_{23} \label{cond2}\\
(\hat q_{23})^2&=&\hat q_{22}\hat q_{13}\\
\hat q_{33}^0&=&\hat q_{13}\hat q_{33}^1=\hat q_{23}\hat q_{33}^2\\
\hat q_{33}^0\hat q_{23}&=&\hat q_{33}^1\hat q_{12}\\
\hat q_{33}^0\hat q_{13}&=&\hat q_{33}^2\hat q_{12}\\
\hat q_{33}^1\hat q_{11}&=&\hat q_{33}^2\hat q_{13}\\
\hat q_{33}^1\hat q_{23}&=&\hat q_{33}^2\hat q_{22}.
\label{cond3}
\end{eqnarray}
Moreover, the action of the group $U(A_4)$ has the form
\begin{eqnarray*}
\hat q_{11}&\rightarrow&\frac{\hat l_2}{(\hat l_1)^2}\hat q_{11}\\
\hat q_{12}=\hat q_{21}&\rightarrow&\frac{\hat l_0}{\hat l_1\hat l_2}\hat q_{21}=\frac{1}{\hat l_1\hat l_2}\hat q_{21}\\
\hat q_{22}&\rightarrow&\frac{\hat l_1}{(\hat l_2)^2}\hat q_{22}\\
\hat q_{13}=\hat q_{31}&\rightarrow&\frac{\hat l_3}{\hat l_1\hat l_3}\hat q_{31}=\frac{1}{\hat l_1}\hat q_{31}\\
\hat q_{23}=\hat q_{32}&\rightarrow&\frac{\hat l_3}{\hat l_3\hat l_2}\hat q_{32}=\frac{1}{\hat l_2}\hat q_{32}\\
\hat q_{33}^0&\rightarrow&\frac{\hat l_0}{(\hat l_3)^2}\hat q_{33}^0=\frac{1}{(\hat l_3)^2}\hat q_{33}^0\\
\hat q_{33}^1&\rightarrow&\frac{\hat l_1}{(\hat l_3)^2}\hat q_{33}^1\\
\hat q_{33}^2&\rightarrow&\frac{\hat l_2}{(\hat l_3)^2}\hat q_{33}^2\\
\textit{and}\,\,\,\,\,[\hat q_{33}^3]&\rightarrow&\frac{1}{\hat l_3}[\hat q_{33}^3]
\end{eqnarray*}
where $\hat l_0=1$ and $\hat l_1,\hat l_2, \hat l_3\in \mathbb{C}^*$. 

Assume that the quantizers are non-zero. 
Then we may chose $\hat l_1,\hat l_2, \hat l_3$ in such a way that all $$\hat q_{11},\hat q_{22},\hat q_{12},\hat q_{21},\hat q_{13},\hat q_{31},\hat q_{23}, \hat q_{32},\hat q^0_{33},\hat q^1_{33},\hat q^2_{33}$$ are equal to 1.

What remains is the $2\times 2$-matrix $[\hat q^3_{33}]$ which by choosing of basis can be transformed to one of the following forms
$$[\hat q^3_{33}]=M=
\left[\begin{array}{cc}
\lambda & 0\\ 
0 &\kappa
\end{array}\right]
,
\left[\begin{array}{cc}
\lambda & 1\\
0 &\lambda
\end{array}\right]$$
where $\lambda, \kappa\in\mathbb{C}.$ 

Hence we have the sequence
\begin{eqnarray}
\begin{array}{c|c|c|c|c|c|c|c|c|c|c|c|c|}
&\hat q_{11}&\hat q_{22}&\hat q_{12}&\hat q_{21}&\hat q_{13}&\hat q_{31}&\hat q_{23}&\hat q_{32} &\hat q_{33}^0&\hat q_{33}^1&\hat q_{33}^2&[\hat q_{33}^3]\\\hline
a)&1&1&1&1&1&1&1&1&1&1&1&M\\\hline
\end{array}
\end{eqnarray} 

Assume now one or more of the quantizers $\hat q_{11},\hat q_{22},\hat q_{12},\hat q_{21},\hat q_{13},\hat q_{31},\hat q_{23},\hat q_{32},\hat q^0_{33},\hat q^1_{33},\hat q^2_{33}$ may be equal zero. The rest will then map to 1 by choosing $\hat l_1,\hat l_2$ properly.

By choosing $\hat l_3$ we reduce the matrix $M$ to one of the following matrices $P$:
$$[\hat q^3_{33}]\rightarrow P=
\left[\begin{array}{cc}
0 & 0\\
0 &0
\end{array}\right]
,
\left[\begin{array}{cc}
0 & 1\\
0 & 0
\end{array}\right]
,
\left[\begin{array}{cc}
1 & 1\\
0 & 1
\end{array}\right]
,
\left[\begin{array}{cc}
1 & 0\\ 
0 &\lambda
\end{array}\right]
.$$

Finally  by the conditions (\ref{cond1}-\ref{cond3}) give the following possible sequences
\begin{eqnarray}
\begin{array}{c|c|c|c|c|c|c|c|c|c|c|c|c|}
&\hat q_{11}&\hat q_{22}&\hat q_{12}&\hat q_{21}&\hat q_{13}&\hat q_{31}&\hat q_{23}&\hat q_{32} &\hat q_{33}^0&\hat q_{33}^1&\hat q_{33}^2&[\hat q_{33}^3]\\\hline
b)&1&1&1&1&1&1&1&1&0&0&0&P\\\hline
c)&0&0&0&0&0&0&0&0&0&1&0&P\\\hline
d)&0&0&0&0&0&0&0&0&0&0&1&P\\\hline
e)&0&0&0&0&0&0&0&0&0&1&1&P\\\hline
f)&0&0&0&0&0&0&0&0&0&0&0&P\\\hline
g)&0&1&0&0&0&0&0&0&0&1&0&P\\\hline
h)&0&1&0&0&0&0&0&0&0&0&0&P\\\hline
i)&1&0&0&0&0&0&0&0&0&0&1&P\\\hline
j)&1&0&0&0&0&0&0&0&0&0&0&P\\\hline
k)&1&1&1&1&0&0&0&0&0&0&0&P\\\hline
\end{array}
\end{eqnarray}

Applying the inverse Fourier transform (\ref{invF}) to the sequences we get the corresponding element in the group algebra,
\begin{eqnarray*}
q&=&\sum_{g,h\in A_4}F^{-1}(\hat q)(g,h)
=\sum_{g,h\in A_4} \sum_{\chi_i,\chi_j\in \hat A_4} d_{i,j} Tr(D^{i,j}(g,h)^*\hat q_{ij})\cdot(g,h)\\
&=&1+\frac{9}{|A_4|^2}\sum_{g,h\in A_4}(Tr(D_{0}^{3,3}(g,h)^*\hat q^0_{33})+Tr(D_{1}^{3,3}(g,h)^*\hat q^1_{33})+Tr(D_{2}^{3,3}(g,h)^* \hat q^2_{33})+Tr(D_{3}^{3,3}(g,h)^*[\hat q^3_{33}]))\cdot(g,h)
\end{eqnarray*}
where, as mentioned, most of the sum cancels out. 

Here $D_{i}^{3,3},i=0,1,2,3$ are components of $D^{3,3}$ operating on the decomposition $E_3\otimes E_3=E_0\oplus E_1\oplus E_2\oplus 2E_3$. 

Writing $M+I$ instead of $M$ we get a shorter form for $q$:
\begin{eqnarray*}
q_a&=&1+\frac{9}{|A_4|^2}\sum_{g,h\in A_4}Tr(D_{3}^{3,3}(g,h)^*M)\cdot(g,h)
\end{eqnarray*}

Further we get
\begin{eqnarray*}
q_b&=&1+\frac{9}{|A_4|^2}\sum_{g,h\in A_4}Tr(D_{3}^{3,3}(g,h)^*P)\cdot(g,h)\\
q_c&=&1+\frac{9}{|A_4|^2}\sum_{g,h\in A_4}(Tr(D_{1}^{3,3}(g,h)^*)+Tr(D_{3}^{3,3}(g,h)^*P))\cdot(g,h)\\
q_d&=&1+\frac{9}{|A_4|^2}\sum_{g,h\in A_4}(Tr(D_{2}^{3,3}(g,h)^*)+Tr(D_{3}^{3,3}(g,h)^*P))\cdot(g,h)\\
q_e&=&1+\frac{9}{|A_4|^2}\sum_{g,h\in A_4}(Tr(D_{1}^{3,3}(g,h)^*)+Tr(D_{2}^{3,3}(g,h)^*)+Tr(D_{3}^{3,3}(g,h)^*P))\cdot(g,h)
\end{eqnarray*}

The combinations f), h), j) and k) give the same quantizer as b), the combinations g) and c), i) and d) give the same quantizers.

We adjust the constants and have the result of the theorem.
\end{proof}

\section*{Appendix}
Let $G=S_3$. By theorem \ref{Qfinal} the quantizers on tensor products of triples of irreducible representations satisfy
\begin{align} 
E_1\otimes E_1\otimes E_2&: &\hat q_{11}\hat q_{02}E_2 &= \hat q_{12}\hat q_{12}E_2,\label{sq1}\\
E_1\otimes E_2\otimes E_1&: &\hat q_{12}\hat q_{21}E_2 &= \hat q_{21}\hat q_{12}E_2,\\
E_2\otimes E_1\otimes E_1&: &\hat q_{21}\hat q_{21}E_2 &= \hat q_{11}\hat q_{20}E_2,\\
E_1\otimes E_2\otimes E_2&: &\hat q_{12}\hat q^0_{22}E_0\oplus\hat q_{12}\hat q^1_{22}E_1\oplus\hat q_{12}\hat q^2_{22}E_2 &= \hat q^1_{22}\hat q_{11}E_0\oplus\hat q^0_{22}\hat q_{10}E_1\oplus\hat q^2_{22}\hat q_{12}E_2,\label{sq2}\\
E_2\otimes E_1\otimes E_2&: &\hat q_{21}\hat q^0_{22}E_0 \oplus  \hat q_{21}\hat q^1_{22}E_1\oplus\hat q_{21}\hat q^2_{22}E_2&= \hat q_{12}\hat q^0_{22}E_0\oplus\hat q_{12}\hat q^1_{22}E_1\oplus\hat q_{12}\hat q^2_{22}E_2,\label{sq3}\\
E_2\otimes E_2\otimes E_1&: & \hat q^1_{22}\hat q_{11}E_0\oplus\hat q^0_{22}\hat q_{01}E_1\oplus\hat q_{21}\hat q^2_{22}E_2 &= \hat q_{21}\hat q^0_{22}E_0\oplus\hat q_{21}\hat q^1_{22}E_1\oplus\hat q_{21}\hat q^2_{22}E_2,\\
E_2\otimes E_2\otimes E_2&: & \hat q^2_{22}\hat q^0_{22}E_0 \oplus\hat q^2_{22}\hat q^1_{22}E_1&= \hat q^2_{22}\hat q^0_{22}E_0\oplus\hat q^2_{22}\hat q^1_{22}E_1\\
and&&\hat q^0_{22}\hat q_{20}E_2\oplus \hat q^1_{22}\hat q_{21}E_2\oplus \hat q^2_{22}\hat q^2_{22}E_2 &= \hat q^0_{22}\hat q_{02}E_2\oplus \hat q^1_{22}\hat q_{12}E_2\oplus \hat q^2_{22}\hat q^2_{22}E_2.
\end{align}
Then by (\ref{sq1})
\begin{eqnarray*}
\hat q_{11} &=& \hat q_{12} \hat q_{12},
\end{eqnarray*}
by (\ref{sq2})
\begin{eqnarray*}
\hat q^0_{22} &=& \hat q_{12}\hat q^1_{22}
\end{eqnarray*}
and by e.g. (\ref{sq3})
\begin{eqnarray*}
\hat q_{12} &=& \hat q_{21}.
\end{eqnarray*}

Let $G=A_4$. By theorem \ref{Qfinal} the quantizers on tensor products of triples of irreducible representations satisfy
\begin{align} 
E_1\otimes E_1\otimes E_1&:&\hat q_{11}\hat q_{12}E_0&=\hat q_{11}\hat q_{21}E_0\\
E_1\otimes E_1\otimes E_2&:&\hat q_{12}\hat q_{10}E_1&=\hat q_{11}\hat q_{22}E_1\label{aq1}\\
E_1\otimes E_2\otimes E_1&:&\hat q_{21}\hat q_{10}E_1&=\hat q_{12}\hat q_{01}E_1\label{aq2}\\
E_2\otimes E_1\otimes E_1&:&\hat q_{11}\hat q_{22}E_1&=\hat q_{21}\hat q_{01}E_1\\
E_1\otimes E_2\otimes E_2&:&\hat q_{22}\hat q_{11}E_2&=\hat q_{12}\hat q_{02}E_2\\
E_2\otimes E_1\otimes E_2&:&\hat q_{12}\hat q_{20}E_2&=\hat q_{21}\hat q_{02}E_2\\
E_2\otimes E_2\otimes E_1&:&\hat q_{21}\hat q_{20}E_2&=\hat q_{22}\hat q_{11}E_2\\
E_1\otimes E_1\otimes E_3&:&\hat q_{13}\hat q_{13}E_3&=\hat q_{11}\hat q_{23}E_3\label{aq3}\\
E_1\otimes E_3\otimes E_1&:&\hat q_{31}\hat q_{13}E_3&=\hat q_{13}\hat q_{31}E_3\\
E_3\otimes E_1\otimes E_1&:&\hat q_{11}\hat q_{32}E_3&=\hat q_{31}\hat q_{31}E_3\label{aq4}\\
E_1\otimes E_3\otimes E_3&:\label{aq7}
&\hat q_{33}^2\hat q_{12}E_0\oplus\hat q_{33}^0\hat q_{10}E_1\oplus\hat q_{33}^1\hat q_{11}E_2\oplus[\hat q_{33}^3]\hat q_{13}2E_3
&=\hat q_{13}\hat q_{33}^0E_0\oplus\hat q_{13}\hat q_{33}^1E_1\oplus\hat q_{13}\hat q_{33}^2E_2\oplus\hat q_{13}[\hat q_{33}^3]2E_3\\
E_3\otimes E_1\otimes E_3&:
&\hat q_{13}\hat q_{33}^0E_0\oplus\hat q_{13}\hat q_{33}^1E_1\oplus\hat q_{13}\hat q_{33}^2E_2\oplus\hat q_{13}[\hat q_{33}^3]2E_3
&=\hat q_{31}\hat q_{33}^0E_0\oplus\hat q_{31}\hat q_{33}^1E_1\oplus\hat q_{31}\hat q_{33}^2E_2\oplus\hat q_{31}[\hat q_{33}^3]2E_3
\label{aq5}\\
E_3\otimes E_3\otimes E_1&:
&\hat q_{31}\hat q_{33}^0E_0\oplus\hat q_{31}\hat q_{33}^1E_1\oplus\hat q_{31}\hat q_{33}^2E_2\oplus\hat q_{31}[\hat q_{33}^3]2E_3
&=\hat q_{33}^2\hat q_{21}E_0\oplus\hat q_{33}^0\hat q_{01}E_1\oplus\hat q_{33}^1\hat q_{11}E_2\oplus[\hat q_{33}^3]\hat q_{31}2E_3\\
E_2\otimes E_2\otimes E_2&:&\hat q_{22}\hat q_{21}E_0&=\hat q_{22}\hat q_{12}E_0\\
E_2\otimes E_2\otimes E_3&:&\hat q_{23}\hat q_{23}E_3&=\hat q_{22}\hat q_{13}E_3\label{aq6}\\
E_2\otimes E_3\otimes E_2&:&\hat q_{32}\hat q_{23}E_3&=\hat q_{23}\hat q_{32}E_3\\
E_3\otimes E_2\otimes E_2&:&\hat q_{22}\hat q_{31}E_3&=\hat q_{32}\hat q_{32}E_3\\
E_2\otimes E_3\otimes E_3&:\label{aq8}
&\hat q_{33}^1\hat q_{21}E_0\oplus\hat q_{33}^2\hat q_{22}E_1\oplus\hat q_{33}^0\hat q_{20}E_2\oplus[\hat q_{33}^3]\hat q_{23}2E_3
&=\hat q_{23}\hat q_{33}^0E_0\oplus\hat q_{23}\hat q_{33}^1E_1\oplus\hat q_{23}\hat q_{33}^2E_2\oplus\hat q_{23}[\hat q_{33}^3]2E_3\\
E_3\otimes E_2\otimes E_3&:
&\hat q_{23}\hat q_{33}^0E_0\oplus\hat q_{23}\hat q_{33}^1E_1\oplus\hat q_{23}\hat q_{33}^2E_2\oplus\hat q_{23}[\hat q_{33}^3]2E_3
&=\hat q_{32}\hat q_{33}^0E_0\oplus\hat q_{32}\hat q_{33}^1E_1\oplus\hat q_{32}\hat q_{33}^2E_2\oplus\hat q_{32}[\hat q_{33}^3]2E_3\\
E_3\otimes E_3\otimes E_2&:
&\hat q_{32}\hat q_{33}^0E_0\oplus\hat q_{32}\hat q_{33}^1E_1\oplus\hat q_{32}\hat q_{33}^2E_2\oplus\hat q_{32}[\hat q_{33}^3]2E_3
&=\hat q_{33}^1\hat q_{12}E_0\oplus\hat q_{33}^2\hat q_{22}E_1\oplus\hat q_{33}^0\hat q_{02}E_2\oplus[\hat q_{33}^3]\hat q_{32}2E_3\\
E_3\otimes E_3\otimes E_3&:
&[\hat q_{33}^3]\hat q_{33}^02E_0\oplus[\hat q_{33}^3]\hat q_{33}^12E_1\oplus[\hat q_{33}^3]\hat q_{33}^22E_2
&=[\hat q_{33}^3]\hat q_{33}^02E_0\oplus[\hat q_{33}^3]\hat q_{33}^12E_1\oplus[\hat q_{33}^3]\hat q_{33}^22E_2\\
and&&\hat q_{33}^0\hat q_{30}E_3\oplus \hat q_{33}^1\hat q_{31}E_3\oplus \hat q_{33}^2\hat q_{32}E_3\oplus [\hat q_{33}^3][\hat q_{33}^3]4E_3
&=\hat q_{33}^0\hat q_{03}E_3\oplus \hat q_{33}^1\hat q_{13}E_3\oplus \hat q_{33}^2\hat q_{23}E_3\oplus [\hat q_{33}^3][\hat q_{33}^3]4E_3
\end{align}
where $[\hat q^3_{33}]$ is a $2\times2$-matrix on the 6 dimensional $2 E_3$. 

Then by (\ref{aq1}) and (\ref{aq2}),
\begin{eqnarray*}
\hat q_{11}\hat q_{22}=\hat q_{12}=\hat q_{21}
\end{eqnarray*}
by (\ref{aq5})
\begin{eqnarray*}
\hat q_{13}=\hat q_{31}
\end{eqnarray*}
by (\ref{aq3})
\begin{eqnarray*}
(\hat q_{13})^2=\hat q_{11}\hat q_{23}
\end{eqnarray*}
which together with (\ref{aq4}) gives 
\begin{eqnarray*}
\hat q_{23}=\hat q_{32}
\end{eqnarray*}
further by (\ref{aq6})
\begin{eqnarray*}
(\hat q_{23})^2&=&\hat q_{22}\hat q_{13}
\end{eqnarray*}
and the last conditions are given by the equations (\ref{aq7}) and (\ref{aq8})
\begin{eqnarray*}
&\hat q_{33}^0=\hat q_{13}\hat q_{33}^1=\hat q_{23}\hat q_{33}^2\\
&\hat q_{33}^0\hat q_{23}=\hat q_{33}^1\hat q_{12}\\
&\hat q_{33}^0\hat q_{13}=\hat q_{33}^2\hat q_{12}\\
&\hat q_{33}^1\hat q_{11}=\hat q_{33}^2\hat q_{13}\\
&\hat q_{33}^1\hat q_{23}=\hat q_{33}^2\hat q_{22}.
\end{eqnarray*}

\end{document}